\documentclass[a4paper,12pt,reqno]{amsart}

\usepackage{amsmath,amssymb,amsthm,amsaddr}
\usepackage[margin=3.4cm,top=3.5cm,bottom=3.8cm,footskip=1cm,headsep=1.5cm]{geometry}
\usepackage[utf8]{inputenc}
\usepackage{graphicx}
\usepackage{xcolor}
\usepackage{doi}

\usepackage[noend]{algpseudocode}
\usepackage[ruled]{algorithm2e}
\usepackage{booktabs}
\usepackage{siunitx}
\usepackage[list=true, font=small, labelfont=bf, 
labelformat=brace, position=top]{subcaption}

\usepackage{enumitem}

\def\RR{\mathbb{R}}

\def\div{\operatorname{div}}
\def\curl{\operatorname{curl}}
\def\hh{\mathbf{h}}
\def\hs{\hh_s}
\def\aa{\mathbf{a}}
\def\bb{\mathbf{b}}
\def\js{\mathbf{j}_s}
\def\nn{\mathbf{n}}
\def\Th{\mathcal{T}_h}
\def\uu{\mathbf{u}}
\def\vv{\mathbf{v}}
\def\JJ{\mathbf{J}}
\def\FF{\mathbf{F}}
\def\AA{\mathbf{A}}
\def\xx{\mathbf{x}}
\def\mmu{\boldsymbol{\mu}}
\def\xxi{\boldsymbol{\xi}}

\def\gg{\mathbf{g}}
\def\ee{\mathbf{e}}
\def\zz{\mathbf{z}}
\def\HH{\mathbf{H}}
\def\yy{\mathbf{y}}
\def\dd{\mathbf{d}}

\newtheorem{lemma}{Lemma}
\newtheorem{theorem}[lemma]{Theorem}

\theoremstyle{definition}

\newtheorem{remark}[lemma]{Remark}
\usepackage{url}

\begin{document}
\title[Nonlinear solvers using Quasi-Newton updates]{On nonlinear magnetic field solvers using local Quasi-Newton updates}

\author{H. Egger$^{1,2}$, F. Engertsberger$^1$, L. Domenig$^3$, K. Roppert$^3$, and M. Kaltenbacher$^3$}

\address{%
\small 
$^1$Institute of Numerical Mathematics, Johannes Kepler University Linz, Austria \\
$^2$Johann Radon Institute for Computational and Applied Mathematics, Linz, Austria
\\
$^3$ Institute of Fundamentals and Theory
in Electrical Engineering, TU Graz, Austria}

\begin{abstract}
Fixed-point or Newton-methods are typically employed for the numerical solution of nonlinear systems arising from discretization of nonlinear magnetic field problems.
We here discuss an alternative strategy which uses local Quasi-Newton updates to construct appropriate linearizations of the material behavior during the nonlinear iteration.  
The resulting scheme shows similar fast convergence as the Newton-method but, like the fixed-point methods, does not require derivative information of the underlying material law. 
As a consequence, the method can be used for the efficient solution of models with hysteresis which involve nonsmooth material behavior. 
The implementation of the proposed scheme can be realized in standard finite-element codes in parallel to the fixed-point and the Newton method. 
A full convergence analysis of all three methods is established proving global mesh-independent convergence. 
The theoretical results and the performance of the nonlinear iterative schemes are evaluated by computational tests for a typical benchmark problem.
\end{abstract}

\maketitle 

\begin{quote}
\end{quote}

\bigskip

\section{Introduction}
\label{sec:intro}

We consider the efficient numerical solution of nonlinear magnetic field problems arising in high-power low-frequency applications like electric machines and transformers \cite{Hrabovcova2020,Salon1995}. The governing equations under consideration are
\begin{alignat}{4}
\curl \hh &= \mathbf{j_s} \quad &\text{in } \Omega, \qquad && \bb &= \bb(\hh), \label{eq:1}\\
\div \bb &= 0 \quad & \text{in } \Omega, \qquad && \bb \cdot \mathbf{n} &= 0 \quad \text{on } \partial\Omega. \label{eq:2}
\end{alignat}
Here $\hh$ is the magnetic field intensity,  $\bb$ the magnetic flux, and $\js = \curl \hs$ the current density, which is represented by a given source field $\hs$. Following~\cite{Silvester91}, we represent the material behavior by a constitutive equation of the form
\begin{align} \label{eq:3}
\bb(\hh) &= \partial_{\hh} w_*(\hh),
\end{align}
with $w_*(\hh)$ the magnetic co-energy density. The material law~\eqref{eq:3} holds for every point $x$ in space and thus has to be understood as $\bb(\xx) = \partial_{\hh} w_*(\xx,\hh(\xx))$.
This relation allows to represent inhomogeneous and anisotropic nonlinear material behavior, including remanent magnetization. As illustrated later on, also the magnetic vector hysteresis models of \cite{Lavet2013,Prigozhin2016} can be phrased in this form.

\subsection*{Numerical solution}
The simulation of \eqref{eq:1}--\eqref{eq:3} is usually based on magnetic scalar or vector potential formulations and their discretization by finite-elements~\cite{Meunier2008}. 
The resulting nonlinear algebraic systems are then tackled by iterative solvers, e.g. fixed-point methods~\cite{Hantila2000}, the Newton-method \cite{Fujiwara2005,Borghi2004}, or certain variants~\cite{Takahashi2013,Dlala2008}.  
All these methods are based on some sort of linearization
\begin{align} \label{eq:4}
\delta \bb \approx \mmu_{loc} \delta \hh
\end{align}
of the nonlinear material law \eqref{eq:3} with $\mmu_{loc}$ representing a local approximation of the field-dependent permeability tensor. 
The Newton method, for instance, uses $\mmu_{loc} = \partial_{\hh} \bb(\hh)$ and thus relies on derivatives of the material law. Together with an appropriate selection of the step size, global linear and local quadratic convergence can be established~\cite{Nocedal2006}.
The choice $\mmu_{loc} = \bar \mmu_{loc}$, independent of the iteration index, results in the fixed-point method~\cite{Hantila2000}. 
Due to the strongly nonlinear behavior of ferromagnetic materials, the convergence of the method may be rather slow in typical applications. This can be partially alleviated by certain modifications~\cite{Takahashi2013,Dlala2007}.
An advantage of fixed-point schemes, on the other hand, is that no information about the derivative $\partial_{\hh} \bb(\hh)$ of the material law is required.
This allows the application to problems where $\bb(\hh)$ is not differentiable which is of relevance in the context of hysteresis~\cite{Lavet2013,Prigozhin2016}.

\subsection*{Scope and main contributions}
In this paper, we discuss an approach 
for choosing the local permeabilities $\mmu_{loc}$ in an iterative solution process that combines the advantages of the fixed-point and the Newton-methods, i.e.: 
\begin{itemize}
\item only access to the material law $\bb(\hh)$ is required;
\item fast convergence, similar to that of the Newton-method, is obtained. 
\end{itemize}
The key idea of the scheme is to employ the low-rank update formulas provided by Quasi-Newton methods, in order to extract significant information about the linearized material behavior, i.e., the  Jacobian $\partial_{\hh} \bb(\hh)$ or generalization of it, from previous values of $\hh$ and $\bb(\hh)$ which are accessible during simulation. 
Let us emphasize that, in contrast to the direct application of Quasi-Newton methods for the solution of large scale nonlinear problems~\cite{Nocedal1980,Geradin1981,Papadrakakis1991}, the Quasi-Newton updates here are performed locally on every element of the mesh, and they are used in order to obtain good candidates for the local permeabilities $\mmu_{loc}$ in the approximation \eqref{eq:4}. 
As a consequence, the resulting schemes can easily be implemented in existing finite-element codes without disturbing the natural sparsity and efficiency of finite-element computations. 

A similar idea has been used in our previous work~\cite{Domenig2024} for the numerical simulation of hysteresis phenomena.
We here consider different update rules and introduce additional modifications to ensure certain constraints on the local permeabilities.
Together with appropriate strategies for choosing the step size in the underlying iterative method, this allows us to establish a complete convergence analysis of the method proving global $r$-linear convergence
with iteration numbers that are independent of the finite-element mesh.
We here present our ideas and proofs in detail for a discretization of the scalar potential formulation of \eqref{eq:1}--\eqref{eq:3} by nodal finite-elements of lowest order. Our main arguments, however, apply also to vector potential formulations and higher order approximations. In principle, they can be generalized to some extent also to nonlinear differential equations arising in different application contexts.

\subsection*{Outline of the article}
In Section~\ref{sec:fem}, we introduce the scalar potential formulation of our problem and discuss a fully practible finite-element discretization.
The iterative solution of the resulting nonlinear systems is investigated in Section~\ref{sec:iter} and global convergence is proven for a wide class of iterative schemes, including the fixed-point and Newton-methods as special cases.
The assumptions required for our analysis are rather general and allow to describe smooth and non-smooth material behavior. 
In Section~\ref{sec:quasi}, we then discuss the application of  Quasi-Newton updates for the construction of local permeability tensors~$\mmu_{loc}$, which can immediately be utilized in the iterative methods of Section~\ref{sec:iter}. This strategy leads to fast, flexible, and reliable solvers for nonlinear finite-element analysis.
The performance of the proposed schemes is demonstrated and compared to alternative methods by numerical tests for a typical benchmark problem considering material models with and without hysteresis.

\section{Discretized scalar potential formulation}
\label{sec:fem}

Let us start with introducing in detail the model problem and its finite-element discretization to be considered in the remainder of the article. 
We assume that $\Omega \subset \RR^d$, $d=2,3$ is a bounded and simply connected Lipschitz domain. Then 
any function $\hh$ satisfying $\curl \hh = \js$ can be written in the form 
\begin{align}
\hh &= \hs - \nabla \psi,
\end{align}
with reduced scalar potential $\psi$, and the problem~\eqref{eq:1}--\eqref{eq:3} can be restated as~\cite{Meunier2008}
\begin{alignat}{2}
\div(\partial_{\hh} w_*(\hs - \nabla \psi)) &= 0 \qquad && \text{in } \Omega, \label{eq:bvp1}\\
\nn \cdot \partial_{\hh} w_*(\hs - \nabla \psi) &= 0 \qquad && \text{on } \partial\Omega. \label{eq:bvp2}
\end{alignat}
This nonlinear system determines $\psi$ uniquely up to a constant~\cite{Engertsberger23}. 

Now let $\Th$ be a triangular respectively tetrahedral finite-element mesh with elements~$T$ and vertices $\xx_i$, $i=0,\ldots,N$. We define the finite-element space
\begin{align}
V_h = \Big\{ v_h = \sum\nolimits_{i=1}^N v_i N_i, \, v_i \in \RR \Big\}
\end{align}
where $N_i$ is the nodal basis function associated with the vertex $\xx_i$~\cite{Braess2007}. 
Since the vertex $\xx_0$ is left out in the summation, the value $v_h(\xx_0)=0$ is fixed by construction for any $v_h \in V_h$. 
The finite-element approximation of the elliptic boundary value problem \eqref{eq:bvp1}--\eqref{eq:bvp2} then is to find $\psi_h \in V_h$ such that 
\begin{align} \label{eq:fem}
\langle \partial_{\hh} w_*(\hs - \nabla \psi_h), \nabla v_h\rangle_h &= 0 \qquad \forall v_h \in V_h.
\end{align}
Here and in the following, we use an approximation
\begin{align} \label{eq:numquad}
\langle \aa,\bb\rangle_h = \sum\nolimits_{T \in \Th} |T| \, \aa(\xx_T) \cdot \bb(\xx_T)
\end{align}
for the scalar product $\langle \aa,\bb\rangle_\Omega=\int_\Omega \aa(\xx) \cdot \bb(\xx) \, d\xx$ obtained by numerical quadrature using the barycenters $\xx_T$ of the elements~$T \in \Th$.
Let us note that 
\begin{align} \label{eq:norm}
\|\nabla \psi_h\|_h := \sqrt{\langle \psi_h,\psi_h\rangle_h}
\end{align}
defines a norm on $V_h$, since $\|\nabla \psi_h\|_h=0$ implies that $\nabla \psi_h = 0$, but by construction $\psi_h(\xx_0)=0$, and hence $\psi_h = 0$ everywhere on $\Omega$. 

The well-posedness of the discrete problem~\eqref{eq:fem} can be guaranteed under rather general assumptions on the problem data~\cite{Engertsberger23}. For completeness  and later reference, we state and prove the corresponding result in detail. 
\begin{theorem} \label{thm:1}
Let 
$\Th$ be a 
regular triangular resp. tetrahedral mesh of $\Omega \subset \RR^d$, $d=2,3$. 
Further let $\hs \in H(\curl;\Omega)$ and $w_* : \Omega \times \RR^d \to \RR$ be piecewise continuous with respect to $\xx$, and assume that for $\xx \in \Omega$ the function $w_*(\xx,\cdot)$,  is continuously differentiable and there exists constants $L,\gamma>0$ such that
\begin{align}
|\partial_{\hh} w_*(\xx,\hh_1) - \partial_{\hh} w_*(\xx,\hh_2)| &\le L |\hh_1-\hh_2| \\
\langle \partial_{\hh} w_*(\xx,\hh_1) - \partial_{\hh} w_*(\xx,\hh_2), \hh_1-\hh_2\rangle &\ge \gamma |\hh_1-\hh_2|^2
\end{align}
for all $\hh_1,\hh_2 \in \RR^d$. 
Then the system \eqref{eq:fem} has a unique solution $\psi_h \in V_h$.
\end{theorem}
\begin{proof}
The finite-element problem \eqref{eq:fem} is equivalent to  the nonlinear system 
\begin{align} \label{eq:sys}
\FF(\uu) = 0 \qquad \text{in } \RR^N 
\end{align}
with $\psi_h(x) = \sum_i u_i N_i(x)$ and $\FF: \RR^N \to \RR^N$ defined by
\begin{align}
\FF_i(\uu) 
&= \langle \partial_{\hh} w_*\big(\hs - \sum\nolimits_{j=1}^N u_j \nabla N_j\big), \nabla N_i \rangle_h. 
\end{align}
It is not difficult to see, that $\FF(\uu) = - \nabla f(\uu)$ where 
\begin{align}
f(\uu) = \langle w_*(\hs - \sum\nolimits_i \uu_i \nabla N_i), 1 \rangle_h
\end{align}
denotes the discrete co-energy functional. From the required properties of $w_*$ we can deduce that $f : \RR^N \to \RR$ is continuously differentiable. 
Moreover 
\begin{align} \label{eq:convex}
f(\tilde \uu) &- f(\uu) - \langle \nabla f(\uu), \tilde \uu - \uu\rangle 
\ge \frac{\gamma}{2} | \tilde \uu - \uu |_h^2.
\end{align}
Here $| \tilde \uu - \uu|_h^2 := \|\nabla \tilde \psi_h - \nabla \psi_h\|^2_h$ with $\psi_h = \sum_i u_i N_i$ and $\tilde \psi_h = \sum_i \tilde u_i N_i$, is the norm on $\RR^N$ induced by the norm $\|\nabla \psi_h\|_h$ on $V_h$. 
The function $f : \RR^N \to \RR$ is thus strongly convex and continuously differentiable, and hence possesses a unique minimizer on $\RR^N$, which is characterized equivalently by the first order optimality condition $\nabla f(\uu)=0$. 
This already implies that \eqref{eq:sys} has a unique solution $\uu \in \RR^N$ which, in turn, corresponds to the unique finite-element solution $\psi_h = \sum_i u_i N_i \in V_h$ of the discretized variational problem~\eqref{eq:fem}.
\end{proof}

\section{Globally convergent iterative solvers}
\label{sec:iter}

For the numerical solution of \eqref{eq:fem}, we consider iterative methods of the form 
\begin{align} \label{eq:iter}
\psi_h^{n+1} = \psi_h^n + \tau^n \delta \psi_h^n, \qquad n \ge 0,
\end{align}
with increments $\delta \psi_h^n \in V_h$ defined by the linear problems 
\begin{align} \label{eq:linfem}
\langle \mmu^n \nabla \delta \psi_h^n, \nabla v_h\rangle_h &= \langle \partial_{\hh} w_*(\hs - \nabla \psi_h), \nabla v_h\rangle_h \qquad \forall v_h \in V_h.
\end{align}
Appropriate choices for the step size $\tau^n>0$ and the local permeabilities $\mmu^n$ allow us to establish the following convergence result for this iterative scheme. 

\begin{theorem} \label{thm:2}
Let the assumptions of Theorem~\ref{thm:1} be valid and $\psi_h \in V_h$ be the unique solution of \eqref{eq:fem}. 
Further let $\mmu^n_T \in \RR^{d \times d}$, $T \in \Th$ be symmetric and assume
\begin{align} \label{eq:mu}
\mu_1 |\xxi|^2 \le \xxi^\top \mmu^n_T \,\xxi \le \mu_2 |\xxi|^2 
\end{align}
for all $T \in \Th$ and $\xxi \in \RR^d$ with constants $\mu_1,\mu_2>0$. 
Then for any $\psi_h^0 \in V_h$ and any stepsize $\tau^n>0$, $n \ge 0$, the iteration \eqref{eq:iter}--\eqref{eq:linfem} is well-defined.
Moreover, there exists a step size rule $\tau^n = \tau(\psi_h^n,\mmu^n)$ such that 
\begin{align} \label{eq:convergence}
\| \nabla (\psi_h^n - \psi_h) \|_h \le C q^n \|\nabla (\psi_h^0 - \psi_h) \|_h, \qquad n \ge 0
\end{align}
with constants $C>0$ and $0<q<1$ only depending on $\gamma$, $L$, $c_1$, and $c_2$.
In particular, the iteration \eqref{eq:iter}--\eqref{eq:linfem} converges globally and $r$-linearly to the unique solution of \eqref{eq:fem}, and the convergence rate is independent of the mesh $\Th$.
\end{theorem}

\subsection*{Sketch of the proof.}
Using the notation introduced in the proof of the previous theorem, we can state the iteration \eqref{eq:iter}--\eqref{eq:linfem} equivalently as 
\begin{align} \label{eq:itersys}
\uu^{n+1} = \uu^n + \tau^n \delta \uu^n
\qquad \text{with} \qquad 
\AA^n \delta \uu^n = - \nabla f(\uu^n)   
\end{align}
with matrix $\AA^n \in \RR^{N \times N}$ defined by $A^n_{ij} = \langle \mmu^n \nabla N_j, \nabla N_i\rangle_h$. 
From this definition and using $\psi_h = \sum_i u_i N_i$, we see that 
\begin{align}
\langle \AA^n \uu, \uu\rangle 
&= \langle \mmu^n \nabla \psi_h, \nabla \psi\rangle_h 
\ge \mu_1 \|\nabla \psi_h\|^2_h = \mu_1 |\uu|^2_h.
\end{align}
Hence $\AA^n$ is symmetric and positive definite, and the iteration is well-defined.
Furthermore, the increment $\delta \uu^n := -(\AA^n)^{-1} \nabla f(\uu^n)$ is a descent direction for minimizing $f(\uu)$.
The global linear convergence of the iterates $\uu^n$ for appropriate stepsize $\tau^n$ to the unique minimizer $\uu$ of $f(\uu)$ then follows from standard results about optimization algorithms~\cite{Nocedal2006}.
A complete proof of the theorem, in particular showing validity of \eqref{eq:convergence} with constants $C$, $q$ that are independent of the mesh $\Th$, is provided in the appendix.

\begin{remark} \label{rem:armijo}
One possible choice of the step size in \eqref{eq:itersys} is $\tau^n=\tau$ constant with $\tau>0$ sufficiently small. In our numerical tests later on, we instead use \emph{Armijo-backtracking}~\cite{Nocedal2006}, i.e.,   
we fix $\tau_{max}>0$ and $0<\sigma,\rho<1$, and choose
\begin{align} \label{eq:armijo}
\tau^n = \max\{&\tau=\tau_{max} \rho^m, \quad m \ge 0, \quad \text{such that}  \\
&f(\uu^n + \tau \delta \uu^n) \le f(\uu^n) + \tau \sigma \langle \nabla f(\uu^n), \delta \uu^n\rangle\}    \notag
\end{align}
As shown in the appendix, this choice is well-defined and allows to apply the assertions of Theorem~\ref{thm:2} leading to global mesh-independent convergence.
\end{remark}

\begin{remark} \label{rem:newton}
The choice $\mmu^n_T = \bar \mmu_T$ with given symmetric positive definite matrices $\bar \mmu_T$ for all $T \in \Th$ leads to a fixed-point iteration with adaptive step size. 
If the co-energy density $w_*(\xx,\hh)$ is twice continuously differentiable in $\hh$, one may also define $\mmu^n_T = \partial_{\hh\hh} w_*(\xx_T,\hs(\xx_T) - \nabla \psi_h^n(\xx_T))$. This yields the Newton-method, again in combination with appropriate line search. 
The global linear convergence of both methods can thus be deduced immediately from Theorem~\ref{thm:2}. 
\end{remark}

\section{A method using local Quasi-Newton updates}
\label{sec:quasi}
We now discuss another choice for local permeability tensors $\mmu^n_T$ in \eqref{eq:linfem} which requires only the evaluation of first derivatives $\partial_{\hh} w_*(\hh)$, but leads to fast convergence of the iterative scheme, similar to that of the Newton method.

\subsection{Quasi-Newton methods}
We start with recalling some results from literature~\cite{Nocedal2006,Deuflhard2011}. 
Let $\gg: \RR^d \to \RR^d$ be continuously differentiable. A Quasi-Newton method aims to solve a nonlinear equation $\gg(\zz)=\mathbf{0}$ by an iterative process
\begin{align}
\zz^{n+1} = \zz^n - \tau^n (\HH^n)^{-1} \gg(\zz^n), \qquad n \ge 0,
\end{align}
with step size $\tau^n>0$ and regular matrices $\HH^n \in \RR^{d \times d}$, $n \ge 0$. 
To ensure fast convergence, the matrices $\HH^n$ should approximate the Jacobian $D \gg(\zz^n)$ in an appropriate manner, i.e., such that the tangent condition
\begin{align}
\HH^{n+1} (\zz^{n+1} - \zz^n)  = \gg(\zz^{n+1})-\gg(\zz^n)
\end{align}
is met. Validity of this condition can be achieved by low-rank update formulas for the matrices $\HH^n$, which can in general be phrased as  
\begin{align}
\HH^{n+1} = \mathcal{QN}(\HH^n,\dd^n,\yy^n).
\end{align}
Here $\dd^n=\zz^{n+1}-\zz^n$ and $\yy^{n}=\gg(\zz^{n+1}) - \gg(\zz^n)$ denote the increments in the arguments and values of the function $\gg$. 
The iteration is initialized by prescribing $\HH^0$ and $\zz^0$. 
Two well-known and widely used Quasi-Newton schemes are the BFGS and DFP methods, whose update formulas are given by
\begin{align*}
\mathcal{QN}_{BFGS}(\HH,\dd,\yy) &= \HH + \frac{\yy \yy^\top}{\yy^\top \dd} - \frac{\HH \dd \dd^\top \HH^\top}{\dd^\top \HH \dd} \\
\mathcal{QN}_{DFP}(\HH,\dd,\yy) &= \HH + \frac{(\yy-\HH \dd) \yy^\top + \yy (\yy-\HH \dd)^\top}{\yy^\top \dd} - \frac{(\yy-\HH \dd)^\top \dd}{(\yy^\top \dd)^2} \yy \yy^\top.
\end{align*}
For properties of the resulting methods, alternative update rules, and further information about Quasi-Newton methods, let us refer to~\cite{Nocedal2006,Deuflhard2011}.

\subsection{Local Quasi-Newton updates}
In the following, we employ the low-rank update formulas of Quasi-Newton methods to define candidates for the local permeabilities $\mmu^n$ in~\eqref{eq:linfem}. Together with some modifications that are required in order to guarantee the assumptions of Theorem~\ref{thm:2}, we obtain the scheme summarized in Algorithm~\ref{alg:nu}.
\begingroup
\IncMargin{1em}
\begin{algorithm}
\textbf{Input:} $\psi_h^0 \in V_h$ and  $\bar \mmu_T \in \RR^{d \times d}$, $T \in \Th$, s.p.d. matrices satisfying \eqref{eq:mu}\\
\textbf{Output:}
Approximate solution $\psi_h$ for \eqref{eq:fem}\\
\smallskip
\tcp{Initialization}
\For{$T\in\Th$}{
set $\mmu^0_T=\bar \mmu_T$, $\hh^0_T = \hs(x_T) - \nabla \psi_h^0(x_T)$, and $\bb^0_T = \partial_{\hh} w_*(\hh_T)$\;
}
\tcp{Iteration}
\For{$n \geq 0$}{
  compute $\delta \psi_h^n$ by solving \eqref{eq:linfem}\; 
  determine step size $\tau^n=\tau(\psi_h^n,\mmu^n)$\; 
  update $\psi_h^{n+1}=\psi_h^n + \tau^n \delta \psi_h^n$\;
  \If{convergence criteria met}{
      \textbf{return} $\psi_h^n$\;
  }
  \For{$T \in \Th$}{
    set $\hh^n_T = \hs(x_T) - \nabla \psi_h^n(x_T)$ and $\bb^n_T = \partial_{\hh} w_*(x_T,\hh^n_T)$\;
    compute $\mmu^{n+1}_T =  \mathcal{QN}(\mmu^n_T,\hh^n_T - \hh^{n-1}_T,\bb^n_T - \bb^{n-1}_T)$\; 
    \tcp{Modifications}
    \If{$\mmu^{n+1}_T$ is not symmetric}{ 
        set $\mmu^{n+1}_T = \frac{1}{2} (\mmu^{n+1}_T + (\mmu^{n+1}_T)^\top)$\;
    }
    \If{\eqref{eq:mu} is not satisfied}{ 
        set $\mmu^n_T = \bar \mmu_T$\;
    }
  } 
} 
\caption{Iteration with local Quasi-Newton updates}
\label{alg:nu}
\end{algorithm}
\endgroup
%
The two modifications after the Quasi-Newton updates allow us to guarantee the conditions on $\mmu^n$ required in Theorem~\ref{thm:2}. 
Alternative modifications guaranteeing \eqref{eq:mu} could be used instead without perturbing validity of our theoretical results; see Section~\ref{sec:num} for an example.

\newpage
\noindent
Together with an appropriate step size choice, we obtain the following result.
\begin{theorem}
\label{thm:3}
Under the assumptions of Theorem~\ref{thm:1}, Algorithm~\ref{alg:nu} is well-defined and the resulting local permeabilities $\mmu^n_T$ obtained during the iteration satisfy~\eqref{eq:mu}. As a consequence, the assertions of Theorem~\ref{thm:2} hold true.
In particular, together with a constant step size $\tau^n = \bar \tau$ sufficiently small, or with Armijo-backtracking~\ref{eq:armijo}, we get global $r$-linear mesh-independent convergence.
\end{theorem}
The claim is a direct consequence of the properties of $\mmu^n$, the assertion of Theorem~\ref{thm:2}, and Remark~\ref{rem:armijo2}. Details can be found in the appendix.
\begin{remark}
Let us emphasize that, compared to the fixed-point or Newton-method, the above algorithm only requires a different routine to setup the local permeability matrices $\mmu_T^n$. The proposed scheme can thus be integrated easily into existing finite-element codes, and the sparsity of the finite-element matrices is not affected. Hence the computational effort for performing one single iteration of the fixed-point, Newton-, or the proposed scheme is very similar. 
\end{remark}

\section{Numerical Results}
\label{sec:num}

To illustrate the performance of the proposed method and to validate the theoretical result of Theorem~\ref{thm:2}, we now present some numerical tests which are inspired by the TEAM benchmark problem 32 \cite{teamproblem32}. 
The setup of this problem justifies the reduction to two space dimensions. The domain $\Omega \subset \RR^2$, depicted in Figure~\ref{fig:1}, thus represents a cross-section of the physical domain. 
\begin{figure}[ht!]
    \centering
    \bigskip
    \includegraphics[trim={2cm 0.8cm 0cm 0.8cm},clip,width=0.49\linewidth]{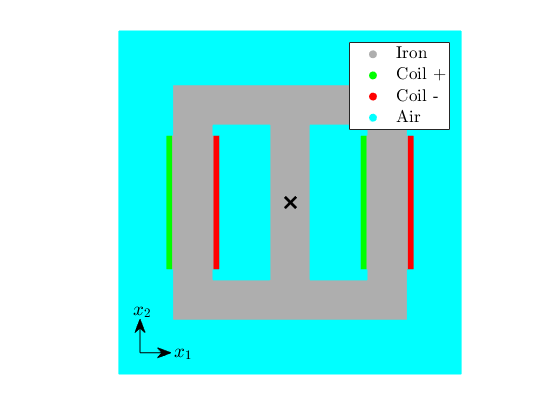}
    \includegraphics[trim={2cm 0.8cm 0cm 0.8cm},clip,width=0.49\linewidth]{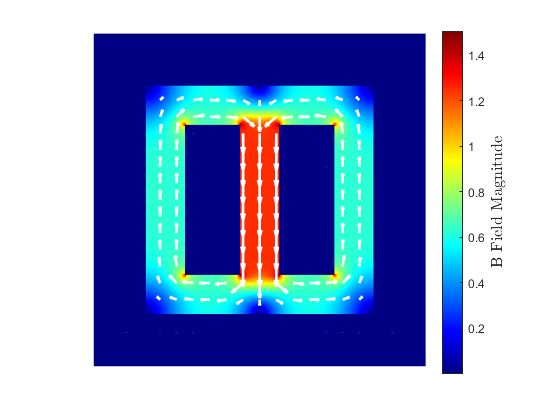}
    \caption{Left: Sketch of the geometry used in our computations with iron (grey), coil (green and red), air (cyan). Field evaluations at the point C6 are considered later on. Right: typical magnetic flux distribution $|\bb| = |\partial_{\hh} w_*(\hh)|$ in our tests.}
    \label{fig:1}
\end{figure}
Further note that the magnetic field and flux here only have in-plane ($x$-$y$) components, while the current density and magnetic vector potential only have a non-trivial $z$-component which can be represented by a scalar field. 
Details about the geometry and the setting can be found in \cite{teamproblem32,Meunier2008}.

\subsection*{General simulation setup}
The domain $\Omega$ is occupied by different materials, i.e., air, copper, and iron; see Figure~\ref{fig:1}.
For copper and air, we choose the linear relation $\bb=\mu_0 \hh$ with $\mu_0 = 4 \pi \, 10^{-7} \, \text{H/m}$ denoting the permeability of vacuum. 
This amounts to a co-energy density $w_*(\hh)=\frac{\mu_0}{2} |\hh|^2$.
To describe the magnetic behavior of the iron, we consider different nonlinear material models with and without hysteresis; details will be specified below. 
The current density is prescribed as a piecewise constant functions with $\js(\xx) = \pm j_s(x_1,x_2) \, \ee_{x_3} \, \text{A/m$^2$}$ on the copper domains and $\js(\xx)=\mathbf{0}$ elsewhere. 
The scalar magnitude $j_s$ of the current will be specified explicitly in our examples.
A piecewise linear source field $\hs$ satisfying $\curl \hs = \js$ here can be constructed analytically.

For our computational tests, we use the finite-element discretization introduced in Section~\ref{sec:fem} over appropriate triangulations of the domain $\Omega$.
All computations are carried out in \textsc{Matlab}.
The iteration \eqref{eq:iter}--\eqref{eq:linfem} is used for the solution of the resulting nonlinear systems. Different choices of the local permeabilities $\mmu^n$ are compared, corresponding to the fixed-point, the Newton, and the proposed scheme with  local Quasi-Newton updates. For the latter method condition \eqref{eq:mu} is not automatically satisfied, hence we enforce it by an eigenvalue decomposition of the $2 \times 2$ tensors $\mmu^n_T$ on every quadrature point and projection of the eigenvalues into the interval $[\mu_1,\mu_2]$. 
This is slightly different to simply setting $\mmu_T^n = \bar \mmu_T$, as stated in Algorithm~\ref{alg:nu}, but does not change the assertions of Theorem~\ref{thm:3}.
Armijo-backtracking \eqref{eq:armijo} with $\sigma=0.1$, $\rho=0.5$, and $\tau_{max} = 1$ is used for choosing the step size~$\tau^n$. 
The iteration \eqref{eq:iter}--\eqref{eq:linfem} is stopped, as soon as the difference of co-energies $|f(\uu^{n+1}) - f(\uu^n)|$ is smaller than $\text{tol}=10^{-8}\,f(\uu^0)$. From \eqref{eq:step5} one can see that $|f(\uu^{n+1}) - f(\uu^n)| \approx |\uu^n - \uu|_h^2 = \|\nabla (\psi_h^n - \psi_h)\|_h^2$, hence this is a good measure for the iteration error.

\subsection{Nonlinear magnetostatics}
The first test case is concerned with a magnetostatic problem without hysteresis. 
As a material law for iron, we choose 
\begin{align}
w_*(\hh) = \frac{\mu_0}{2} |\hh|^2 + U_*(\hh)
\end{align}
with internal energy density $U_*(\hh) = \frac{2 \JJ_s}{\pi}\Big( |\hh| \arctan\big( \frac{|\hh|}{A} \big) - \frac{1}{2} A \log\big( A^2 + |\hh|^2 \big) \Big)$. 
This amounts to the nonlinear material law
\begin{align}
    \label{eqn:material:law:anhysteretic}
    \bb = \partial_{\hh} w_*(\hh) = \mu_0  \hh + \frac{2 \JJ_s}{\pi} \arctan\Big(\frac{| \hh |}{A}\Big) \frac{\hh}{| \hh |} ,
\end{align}
which is frequently applied in the literature~\cite{Prigozhin2016}.
Let us note that this function satisfies the assumptions of Theorem~\ref{thm:2} with $\gamma=\mu_0$ and $L = \mu_0 + \frac{2 \JJ_s}{\pi A}$.
Moreover, the co-energy density $w_*(\hh)$ here is twice differentiable with  
\begin{align*}
\partial_{\hh\hh} w_*(\hh)
= \mu_0 I + \frac{2 \JJ_s}{\pi} \Big(\frac{\arctan(| \hh |/A)}{|\hh|} (I - \frac{\hh \, \hh^\top}{|\hh|^2}) + \frac{1}{A (1 + \frac{|\hh|^2}{A^2})} \frac{\hh\,\hh^\top}{|\hh|^2}\Big) 
\end{align*}
One can see that $\mmu^n = \partial_{\hh} \bb(\hh)$ satisfies \eqref{eq:mu} with $\mu_1 = \gamma$ and $\mu_2= L$, as specified above, which allows to show convergence of the Newton method. 

\subsection*{Numerical results}
For our computations, we choose 
the material parameters as $A=90.302 \, \text{A/m}$ and $\JJ_s=1.5733 \, \text{T}$, and we consider $j_s= 1 \times 10^5 \, \text{A/m$^2$}$ as amplitude of the excitation current.  
In Table~\ref{tab:1}, we compare the performance of the Newton-method, the fixed-point iteration, and the scheme discussed in Section~\ref{sec:quasi} using two different choices of local Quasi-Newton updates. 
For the fixed point method we use the constant tensor $ \bar \mmu_T = \mu_0 \mathbf{I}$ for all $T \in \Th$.
\begin{table}[ht!]
\centering
\setlength\tabcolsep{1.5ex}
\renewcommand{\arraystretch}{1.1}
\begin{tabular}{c||ccccc} 
 & \multicolumn{5}{c}{Iterations} \\
dof & Newton & Fixpoint & BFGS & DFP &   \\
\hline
\hline
$1477$ &  5 & 29 & 12 & 11 &   \\
$5789$ &  5 & 31 & 12 & 11 &   \\
$22921$ &  5 & 30 & 17 & 11 &   \\
$91217$ &  5 & 33 & 17 & 11 &   \\
\end{tabular}
\medskip
\caption{Iteration numbers for different methods and refinement levels for problem with nonlinear anhysteretic material law.}
\label{tab:1}
\end{table}

Note that in perfect agreement with Theorem~\ref{thm:2}, the number of iterations required for the different methods is independent of the mesh size. 
Obviously, the Newton-method converges with the fewest iterations which also perfectly meets our expectations.
While the convergence of the fixed-point iteration is substantially slower, the performance of the two methods using local Quasi-Newton updates is rather similar to that of the Newton method. 
Recall that these methods only use evaluations of the material law $\bb=\bb(\hh)$, but not of the derivative $\partial_{\hh} \bb(\hh)$ as required by the Newton-method. 
The bounds~\eqref{eq:mu} are automatically satisfied for the Newton- and the fixed-point method. For the two schemes using local Quasi-Newton updates, a truncation of the local permeabilities $\mmu_T^n$ was required from time- to time in order to satisfy the bounds~\eqref{eq:mu} and thus guarantee the global mesh-independent convergence of the methods.

\subsection{A vector hysteresis model}
\label{subsec:hysteresis}
In the second example, we use a vector hysteresis model to represent the material behavior of iron. To fit into our setting, we start with defining the co-energy density by 
\begin{align}
\label{eqn:co-energy:hysterese}
w_*(\hh) = \frac{\mu_0}{2} |\hh|^2 + \sum\nolimits_{k = 1}^M w_k \sup\nolimits_{\JJ_k} \Big(\langle \hh, \JJ_k \rangle - U(\JJ_k) - \chi_k |\JJ_k - \JJ_{k,p}|\Big) 
\end{align}
with $U(\JJ) = -\frac{2 A \JJ_s}{\pi} \log(\cos(\frac{\pi}{2}\frac{\JJ}{\JJ_s}))$ and $w_k,\,\chi_k  \geq 0$. The maximization problems in the second term have unique solutions $\JJ_k =\JJ_k(\hh)$, so the co-energy density $w_*(\hh)$ is well-defined.
One can further show that \eqref{eqn:co-energy:hysterese} is continuously differentiable and by Danskin's theorem~\cite{Danskin1966}, the derivative is given by
\begin{align}
\label{eqn:material:law:hysterese}
\bb = \partial_{\hh} w_*(\hh) = \mu_0 \hh + \sum\nolimits_{k = 1}^M  w_k \, \JJ_k(\hh).
\end{align}
This exactly amounts to the vector hysteresis model proposed in~\cite{Lavet2013}. 
%
The evaluation of $w_*(\hh)$ and $\partial_{\hh} w_*(\hh)$ here involves the solution of inner maximization problems which, however, can be done efficiently in a numerical implementation~\cite{Prigozhin2016,Kaltenbacher2022}. 
The constitutive law \eqref{eqn:material:law:hysterese} can further be shown to be Lipschitz continuous and strongly monotone, such that the assumptions of Theorem~\ref{thm:1} are satisfied. 
Due to the non-smooth term in~\eqref{eqn:co-energy:hysterese}, it is however not differentiable, so that the classical Newton method cannot be applied in general. Semismooth Newton methods, however, may still be applicable~\cite{Ulbrich2011,Willerich2019}.
%

\begin{remark}
For $M = 1, w_k = 1, \chi_1 =0$, the solution $\JJ=\JJ_1$ of the inner maximization problem is given by $
\JJ(\hh) = \partial_{\hh} U_*(\hh) = \frac{2 \JJ_s}{\pi} \arctan\Big(\frac{| \hh |}{A}\Big) \frac{\hh}{| \hh |}.
$
In this case, the model \eqref{eqn:material:law:hysterese} exactly amounts to the material law \eqref{eqn:material:law:anhysteretic} considered before.
The above formulas thus provide a natural generalization of anhysteretic material models to the hysteretic case. 
\end{remark}

\subsection*{Numerica results.}
For our computations, we consider the model \eqref{eqn:co-energy:hysterese} with parameters $\chi_k, w_k$ taken from \cite{Domenig2024}.
The magnetic polarizations are set to $\JJ_{k,p} = 0$.
We otherwise choose the same simulation setup as in the previous example and again use $j_s= 1\times 10^5 \, \text{A/m$^2$}$ as amplitude for the excitation current. 
The iterative scheme \eqref{eq:iter}--\eqref{eq:linfem} with local permeabilites $\mmu^n_T=\bar \mmu_T$ fixed or defined by local Quasi-Newton updates is applied. 
In Table~\ref{tab:2}, we summarize the corresponding results.

\begin{table}[ht!]
\centering
\setlength\tabcolsep{1.5ex}
\renewcommand{\arraystretch}{1.1}
\begin{tabular}{c||ccccc} 
 & \multicolumn{4}{c}{Iterations} \\
dof & Fixpoint & BFGS & DFP &   \\
\hline
\hline
$1477$ & 53 & 14 & 10 &   \\
$5789$ &  58 & 14 & 10 &   \\
$22921$ &  55 & 14 & 11 &   \\
$91217$ &  58 & 16 & 11 &   \\
\end{tabular}
\caption{Iteration numbers for different methods and refinement levels for the test problem with hysteresis model.}
\label{tab:2}
\end{table}
We observe iteration numbers, which are comparable to the anhysteretic results in Table \ref{tab:1}. 
Similar to the previous examples, a truncation of the local permeabilities $\mmu_T^n$ was required sometimes for the schemes using Quasi-Newton updates in order to satisfy the bounds~\eqref{eq:mu} required for our analysis.

\subsection{Simulation of a typical load cycle}

As a last example, we deal with a typical situation arising in practice. Here we consider the excitation by a sequence of load currents $\js = \js(t_i)$, $i=1,\ldots,N$, which are taken from test case~2 of TEAM~problem~32~\cite{teamproblem32}. 
We are then mainly interested in the average performance of the proposed methods over the full load cycle.
%

\subsection*{Anhysteretic material model}
In the absence of hysteresis, the corresponding magnetostatic problems can be solved independently from each other. 
Since the excitation currents are smoothly varying in time, we may however choose the solution of the previous time step as the initial guess for the next computation. 
In Table~\ref{table:anhysteretic:multiple:timesteps}, we report about the corresponding computations. 
\begin{table}[ht!]
\centering
\setlength\tabcolsep{1.5ex}
\renewcommand{\arraystretch}{1.1}
\begin{tabular}{c||ccccc} 
 & \multicolumn{5}{c}{Average iterations} \\
dof & Newton & Fixpoint & BFGS & DFP &   \\
\hline
\hline
$1477$ &  3.5 & 13.9 & 3.6 & 3.5&   \\
$5789$ &  3.5 & 15.4 & 3.6 & 3.5 &   \\
$22921$ &  3.5 & 16.2 & 3.6 & 3.5 &   \\
$91217$ &  3.5 & 17.2 & 3.6 & 3.6 &   \\
\end{tabular}
\caption{Average iteration numbers for solution of multiple loading steps (test case~2 of TEAM~problem~32 with $N = 402$ time steps) with anhysteretic material model. Results correspond to different methods and refinement levels.}
\medskip
\label{table:anhysteretic:multiple:timesteps}
\end{table}
Let us note that the average iteration numbers are smaller than in the single loading step. This is the expected behavior, since the solution of the previous time step was used as initial guess for the next one, i.e., the iterative solvers at later time steps start already start in the vicinity of the solution.
Similar to before, the performance of the Quasi-Newton methods is comparable to that of the Newton method, while the fixedpoint method requires significantly more iterations.

\subsection*{Model with hysteresis}
We now repeat the computations of the previous example with the energy-based hysteresis model~\eqref{eqn:material:law:hysterese}.
For the computation at time step $t_i$, the value $\JJ_{k,p}(t_i)=\JJ_k(t_{i-1})$ here corresponds to the magnetic polarization computed in the previous time step. 
The results of our computations are summarized in Table~\ref{table:hysteresis:multiple:timesteps}.
\begin{table}[ht!]
\centering
\setlength\tabcolsep{1.5ex}
\renewcommand{\arraystretch}{1.1}
\begin{tabular}{c||ccccc} 
 & \multicolumn{4}{c}{Average iterations} \\
dof & Fixpoint & BFGS & DFP &   \\
\hline
\hline
$1477$  & 32.3 & 6.7 & 7.2 &   \\
$5789$  & 47.7 & 6.8 & 7.3 &   \\
$22921$ & 53.9 & 6.9 & 7.4 &   \\
$91217$ & 52.1 & 7.0 & 7.4 &   \\
\end{tabular}
\caption{Average iteration numbers for solution of multiple loading steps (test case~2 of TEAM~problem~32 with $N = 402$ time steps) with a hysteretic material model. Results are compared for different methods and refinement levels.}
\label{table:hysteresis:multiple:timesteps}
\end{table}
In comparison to the simulations for the anhysteretic case, the iteration numbers are approximately doubled for all methods under investigation. 
Due to the good initialization at the later time steps, the average iteration numbers are again significantly lower than for the simulations of a single load step.
\begin{figure}[ht!]
\includegraphics[width=0.49\textwidth]{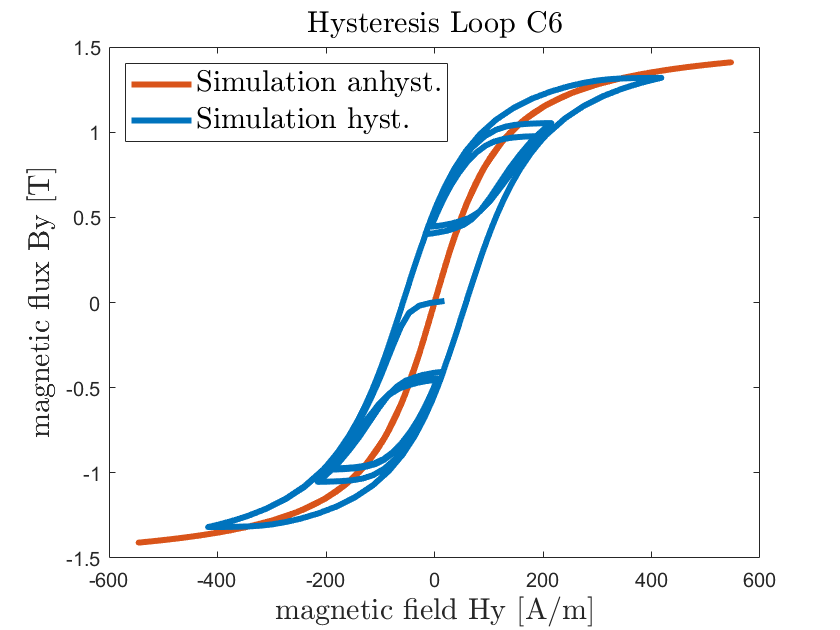}
\includegraphics[width=0.49\textwidth]{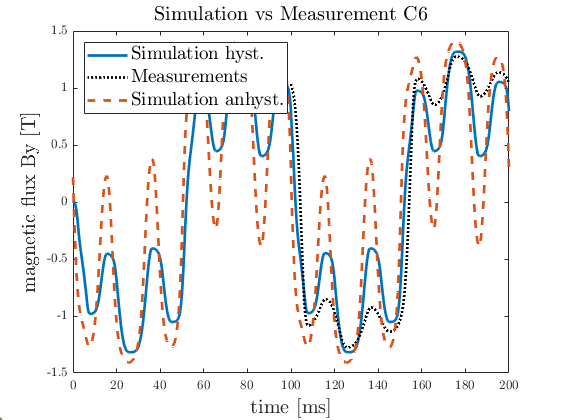}
\caption{Simulation results at the point C6 (see Figure~\ref{fig:1}) for case~2 of TEAM~problem~32. Left: Hysteresis loop for the hysteretic and anhysteretic material model. Right: Comparison of the hysteretic and anhysteretic model with the measurements~\cite{teamproblem32}.\label{fig:2}}
\end{figure}
In Figure~\ref{fig:2}, we compare the vales of the relevant field components 
at a certain point (C6 in Figure~\ref{fig:1}) obtained during simulation of the load cycle for the models with and without hysteresis.
The left plot clearly visualizes the anhysteretic curve as well as the hysteresis loop for the two material laws under investigation. 
The right plot displays the relevant component of the magnetic induction field. For comparison, also corresponding measurement results from~\cite{teamproblem32} are displayed. 
The incorporation of hysteresis clearly improves the accuracy of the simulation.

\section{Discussion}
In this paper, we studied the use of Quasi-Newton updates to obtain approximations of the local magnetic permeabilities during nonlinear finite-element computations. Together with an appropriate choice of the step size, global convergence of the resulting iterative schemes could be proven. 
The methods and proofs are applicable to material models with and without hysteresis. In the smooth setting, we observed a similar fast convergence behavior as for the Newton method which, however, is not applicable for models with hysteresis.
The proposed scheme using local Quasi-Newton updates is applicable to models with and without hysteresis. In our computations, we observed significantly faster convergence than for the fixed-point iteration. A detailed comparison with other fixed-point schemes, see. e.g., \cite{Lavet2013,Prigozhin2016,Hantila2000,Dlala2008hys,Jacques2018}, will be provided elsewhere.

In the current work, we derived a full convergence analysis for the iterations using local Quasi-Newton updates based on a low-order finite-element approximation of the scalar potential formulation of magnetostatics. The main arguments, however, also apply to vector potential formulations and higher order approximations~\cite{Egger2024vec,Friedrich2019}. 
After appropriate time discretization, the algorithms may further be useful for the solution of nonlinear eddy current problems~\cite{Acevedo2023,Bermudez2015,Slodicka2006}, with applications in induction heating~\cite{Chovan2017,Clain1993} and the simulation of superconductivity~\cite{Dular2020,Slodicka2008}.
Investigations in these directions are left for future research.

\begingroup
\footnotesize
\subsection*{Acknowledgements}
This work was supported by the joint DFG/FWF Collaborative Research Centre CREATOR (DFG: Project-ID 492661287/TRR 361; FWF: 10.55776/F90) at TU Darmstadt, TU Graz and JKU Linz.
\endgroup



\newpage 

\appendix
\section{Proof of Theorem~\ref{thm:2}}
For completeness, we now present a detailed proof of Theorem~\ref{thm:2}. 
The main challenge is to show that the constants $C$ and $q$ in \eqref{eq:convergence} are independent of the mesh~$\Th$. 
To show this, we use an analysis on the continuous level, i.e., involving the finite-element functions $\psi_h(x)=\sum_i u_i N_i(x)$ rather than the vectors $\uu \in \RR^N$ arising in the implementation. 
The proof will be split into several steps.

\subsection*{Step~1.}
We denote by $g(v_h) := \langle w_*(\hs-\nabla v_h), 1\rangle_h$ the total co-energy associated with the finite-element function $v_h \in V_h$. Comparing with the notation used in the proof of Theorem~\ref{thm:1}, we see that $g(v_h) = f(\vv)$ where $\vv$ is the coordinate vector associated with $v_h=\sum_i v_i N_i$.
Hence $g : V_h \to \RR$ is strongly convex, and the solution $\psi_h$ of~\eqref{eq:fem} corresponds to the unique minimizer of $g$ over $V_h$. 

\subsection*{Step~2.}
To measure the distance between $v_h \in V_h$ and the solution $\psi_h$ of \eqref{eq:fem}, we consider the energy difference $g(v_h|\psi_h) := g(v_h) - g(\psi_h)$. Since $\psi_h$ is the unique minimizer of $g$ over $V_h$, we see that $g(v_h|\psi_h) \ge 0$ and $g(v_h|\psi_h)=0$ if, and only if, $v_h = \psi_h$. In addition, we have 
\begin{align} \label{eq:step2}
\frac{\gamma}{2} \|\nabla (v_h - \psi_h)\|_h^2 \le g(v_h|\psi_h) \le \frac{L}{2} \|\nabla (v_h - \psi_h)\|_h^2 \qquad \forall v_h \in V_h,
\end{align}
which can be seen as follows: By the fundamental theorem of calculus, we get
\begin{align*}
&w_*(\bb) - w_*(\aa) 
= \int_0^1 \langle \partial_{\hh} w_*(\aa + t (\bb-\aa)), \bb-\aa \rangle \, dt  \\
&= \int_0^1 \langle \partial_{\hh} w_*(\aa + t (\bb-\aa)) - \partial_h w_*(\aa), \bb-\aa \rangle \, dt + \langle \partial_{\hh} w_*(\aa), \bb-\aa\rangle.
\end{align*}
These formulas concern a single quadrature point $x_T$ and $\langle \aa,\bb \rangle = \aa \cdot \bb$ is the Euclidean scalar product on $\RR^d$.
The inequalities in \eqref{eq:step2} then follow by applying this identity with $\aa=\hs - \nabla \psi_h$ and $\bb=\hs-\nabla v_h$ for every point $\xx_T$, employing the properties of $\partial_{\hh} w_*$, summing over $T \in \Th$, and using \eqref{eq:fem}.

\subsection*{Step~3.}
The left hand side of \eqref{eq:linfem} induces a scalar product 
\begin{align} \label{eq:sp}
\langle u_h,v_h\rangle_{1,\mmu^n} := \langle \mmu^n \nabla u_h, \nabla v_h\rangle_h
\end{align} 
on the finite-element space $V_h$, and we denote by $\|v_h\|_{1,\mmu^n} = \sqrt{\langle v_h, v_h\rangle_{1,\mmu^n}}$ the associated norm. From the bounds \eqref{eq:mu} on $\mmu^n$, we immediately see that 
\begin{align} \label{eq:step3}
\mu_1 \|\nabla v_h\|^2_h \le \|v_h\|^2_{1,\mmu^n} \le \mu_2 \|\nabla v_h\|_h^2 \qquad \forall v_h \in V_h.
\end{align}

\subsection*{Step~4.}
Let $\delta \psi_h^n$ denote the solution of \eqref{eq:linfem}. 
Then we can show that 
\begin{align} \label{eq:step4}
\|\delta \psi_h^n\|_{1,\mmu^n}^2 \ge \frac{\gamma^2}{\mu_2} \|\nabla (\psi_h^n - \psi_h)\|_h^2.
\end{align}
To see this, we use a well-known relation between scalar products and norms, which together with \eqref{eq:sp} yields the identity
\begin{align*}
\|\delta \psi_h^n\|_{1,\mmu^n} 
&= \sup_{v_h \in V_h} \frac{\langle  \delta \psi_h^n, v_h \rangle_{1,\mmu^n}}{\|v_h\|_{1,\mmu^n}} 
= \sup_{v_h \in V_h} \frac{\langle  \mmu^n \nabla \delta \psi_h^n, \nabla v_h \rangle_{h}}{\|v_h\|_{1,\mmu^n}}.
\end{align*}
By using \eqref{eq:linfem} and \eqref{eq:fem}, and choosing $v_h = \psi_h^n - \psi_h$, we then further see that 
\begin{align*}
\|\delta \psi_h^n\|_{1,\mmu^n} 
&\ge \frac{\langle \partial_{\hh} w_*(\hs - \nabla \psi_h^n) - \partial_{\hh} w_*(\hs - \nabla \psi_h), \nabla \psi_h^n - \nabla \psi^h\rangle_h}{\|\psi_h^n - \psi_h\|_{1,\mmu^n}}.
\end{align*}
The inequality \eqref{eq:step4} now follows from the strong monotonicity of $\partial_{\hh} w_*(\cdot)$ assumed in Theorem~\ref{thm:1} and employing the right bound in \eqref{eq:step3}. 

\subsection*{Step~5.}
Using similar arguments as in Step~2 and Step~4, we obtain
\begin{align} 
g(\psi_h^n + \tau \delta \psi_h^n|\psi_h^n) 
&\le -\tau \|\delta \psi_h^n\|^2_{1,\mmu^n} + \tfrac{\tau^2}{2} L \|\delta \psi_h^n\|^2_{h} \notag \\
&\le -\tau (1-\tfrac{\tau}{2} \tfrac{L}{\mu_1}) \, \|\delta \psi_h^n\|^2_{1,\mmu^n}.\label{eq:step5}
\end{align}
Hence the total energy is strictly decreasing if $\delta \psi_h^n \ne 0$ and if we choose the step size $\tau$ sufficviently small, i.e., such that $0<\tau<\frac{2 \mu_1}{L}$.

\subsection*{Step~6.}
By combination of the previous estimates, we immediately obtain
\begin{align*}
g(\psi_h^{n+1}|\psi_h) 
&= g(\psi_h^{n+1}|\psi_h^n) + g(\psi_h^n|\psi_h) \\
&\le -\tau^n (1-\tfrac{\tau^n}{2} \tfrac{L}{\mu_1}) \|\delta \psi_h^n\|_{1,\mmu^n}^2 + g(\psi_h^n|\psi_h) \\
&\le (1 - \tau^n (1-\tfrac{\tau^n}{2} \tfrac{L}{\mu_1}) \tfrac{\gamma^2}{\mu_2} \tfrac{L}{2})  \, g(\psi_h^n|\psi_h). 
\end{align*}
This shows that $g(\psi_h^{n+1}|\psi_h^n) \le q(\tau^n) \, g(\psi_h^n|\psi_h)$ with $q(\tau)=1 - \tau (1-\tfrac{\tau}{2} \tfrac{L}{\mu_2}) \tfrac{\gamma^2}{2}\tfrac{L}{2}$.  
Let us note that $0 < q(\tau) < 1$ for $0 < \tau < \frac{2 \mu_2}{L}$, i.e., the energy distance to the solution $\psi_h$ is decreasing for appropriate choice of the step size $\tau^n$.

\subsection*{Step~7.}
We fix $0<\bar \tau < \frac{2 \mu_2}{L}$ and set $\tau^n=\bar \tau$ for all steps of the iteration. By applying the above inequality recursively, we then get
\begin{align*}
g(\psi_h^n|\psi_h) \le q^n g(\psi_h^0|\psi_h)
\end{align*}
with contraction factor $q = q(\bar \tau) < 1$, which only depends on $\gamma$, $L$, $\mu_1$, $\mu_2$, and our choice of $\bar \tau$. 
Application of \eqref{eq:step3} then leads to \eqref{eq:convergence} with this factor $q$ and constant $C=\frac{\mu_2}{\mu_1}$. 
This completes the proof of Theorem~\ref{thm:2}. \qed

\begin{remark} \label{rem:armijo2}
From inequality \eqref{eq:step5} one can deduce that the Armijo-backtracking strategy \eqref{eq:armijo} yields a step size $\tau^n$ satisfying $2 \rho (1-\sigma) \frac{\mu_1}{L} < \tau^n \le \tau_{max}$. By minor modification of the arguments in Step~6 and 7, this again leads to global linear convergence $0<q<1$ independent of $\Th$; see~\cite{Engertsberger23} for details.
\end{remark}

\end{document}